\documentclass[11pt]{amsart}
\usepackage[utf8]{inputenc}
\usepackage{fontenc}
\usepackage{amsfonts}
\usepackage{amssymb}
\usepackage{amsmath}
\usepackage{amsthm}\usepackage{mathtools}
\usepackage{enumerate}
\usepackage{hyperref}\usepackage{enumitem}
\usepackage{mathrsfs}
\usepackage{tikz}\usetikzlibrary{calc}
\usepackage{marginnote}
\usepackage{xcolor,enumitem}
\usepackage{soul}\usepackage{tikz}
\usepackage[all,cmtip]{xy}

\newcommand{\mathscripty}{\mathscr}

\newcommand{\e}{\varepsilon}

\newcommand{\Z}{\mathbb{Z}}

\newcommand{\N}{\mathbb{N}}
\newcommand{\C}{\mathbb{C}}

\newcommand{\eps}{\varepsilon}

\theoremstyle{Case1}

\theoremstyle{Case2}

\newcommand{\NN}{\mathbb{N}}

\newcommand{\lproe}{\mathrm{B}^p_u}
\newcommand{\loneroe}{\mathrm{B}^1_u}

\newcommand{\cstu}{\mathrm{C}^*_u}

\newtheorem*{rigprob*}{Rigidity Problem for uniform Roe Algebras}
\newtheorem*{rigprobcorona*}{Rigidity Problem for uniform Roe Coronas}

\newcommand{\SI}{\mathscripty{I}}
\newcommand{\SJ}{\mathscripty{J}}

\newcommand{\cstar}{$\mathrm{C}^*$}

\newcommand{\cP}{\mathcal{P}}

\newcommand{\cB}{\mathcal{B}}
\newcommand{\cK}{\mathcal{K}}

\newtheorem{theorem}{Theorem}[section]
\newtheorem*{theorem*}{Theorem}
\newtheorem{proposition}[theorem]{Proposition}

\newtheorem*{proposition*}{Proposition}
\newtheorem{lemma}[theorem]{Lemma}
\newtheorem*{lemma*}{Lemma}
\newtheorem{corollary}[theorem]{Corollary}
\newtheorem*{corollary*}{Corollar}

\newtheorem*{fact*}{Fact}
\theoremstyle{definition}
\newtheorem{definition}[theorem]{Definition}
\newtheorem*{definition*}{Definition}
\newtheorem{claim}[theorem]{Claim}
\newtheorem*{claim*}{Claim}

\newtheorem*{conjecture*}{Conjecture}

\theoremstyle{remark}

\newtheorem*{example*}{Example}

\newtheorem*{remark*}{Remark}

\newtheorem*{note*}{Note}
\newtheorem*{question*}{Question}

\newcommand{\norm}[1]{\left\lVert #1 \right\rVert}

\DeclareMathOperator{\propg}{prop}

\DeclareMathOperator{\rank}{rank}

\DeclareMathOperator{\diam}{diam}

\newcounter{my_enumerate_counter}
\newcommand{\pushcounter}{\setcounter{my_enumerate_counter}{\value{enumi}}}
\newcommand{\popcounter}{\setcounter{enumi}{\value{my_enumerate_counter}}}

\usepackage{enumitem}

\begin{document}

\title[Embeddings of $\ell_p$ Roe algebras]{On the uniform Roe algebra as a Banach algebra and embeddings of $\ell_p$ uniform Roe algebras}%

\author[B. M. Braga]{Bruno M. Braga}
\address[B. M. Braga]{Department of Mathematics and Statistics,
York University,
4700 Keele Street,
Toronto, Ontario, Canada, M3J
1P3}
\email{demendoncabraga@gmail.com}
\urladdr{https://sites.google.com/site/demendoncabraga}

\author[A. Vignati]{Alessandro Vignati}
\address[A. Vignati]{ Department of Mathematics, KU Leuven, Celestijnenlaan 200B, B-3001 Leuven, Belgium
}
\email{ale.vignati@gmail.com}
\urladdr{http://www.automorph.net/avignati}

\subjclass[2010]{}
\thanks{}
\date{\today}%
\maketitle

\begin{abstract}
We work on $\ell_p$ uniform Roe algebras associated to metric spaces, and on their mutual embedding. We generalize results of I. Farah and the authors to mutual embeddings of uniform Roe algebras of operators on $\ell_p$ spaces. Simultaneously, we obtain rigidity results for the classic  uniform Roe \cstar-algebras which depend only on their   Banach algebra structure.
\end{abstract}

\setcounter{tocdepth}{1}

\section{Introduction}
Given a metric space $(X,d)$, one defines its uniform Roe \cstar-algebra $\cstu(X)$ as the subalgebra of $\mathcal B(\ell_2(X))$ given by the closure of the algebra of operators having finite propagation with respect to the distance $d$ (we refer the reader to Section \ref{section:preliminaries} for precise definitions). These algebras, introduced by J. Roe (see \cite{RoeBook}), are capable of encoding in algebraic terms some of the large scale geometric properties of $X$.  Their study was boosted in the last two decades, motivated by their connections with the Baum-Connes and the Novikov conjectures (\cite{Yu2000,HigsonRoe1995}), and therefore with geometric group theory, coarse geometry (\cite{GuentnerKaminker2004,HigsonRoe1995,Liliao2017,Ozawa2000,SkandalisTuYu2002}) and, recently, the study of topological phases of matter (\cite{Kubota2017}). 

In this paper, we study the generalization of uniform Roe \cstar-algebras to algebras of operators on $\ell_p$, for $p\in[1,\infty)$, and their mutual embeddings, by generalizing to the case $p\neq 2$ the results obtained for uniform Roe \cstar-algebras in \cite{BragaFarahVignati2019}. Since we only treat such objects from a Banach algebraic point of view, as a consequence we are able to remove certain hypotheses from results of I. Farah and the authors (\cite{BragaFarah2018} and \cite{BragaFarahVignati2019}), J. Spakula and R. Willett (\cite{SpakulaWillett2013AdvMath}), and Y. Chung and K. Li (\cite{ChungLi2018}).

The interest on the $\ell_p$ version of the uniform Roe \cstar-algebra arose in the work of J. Spakula and R. Willett, who in \cite{SpakulaWillett2017} introduced the $\ell_p$ uniform Roe algebra of $X$, denoted $\lproe(X)$, in connection with criteria for Fredholmness. Later, Y. Chung and K. Li  studied in \cite{ChungLi2018} the problem on how bijective coarse equivalence of metric spaces is related to isometric isomorphism of $\ell_p$ uniform Roe algebras.  They generalized results obtained for uniform Roe \cstar-algebras to the case $p\neq 2$: using Lamperti's Theorem \cite[Theorem 3.1]{Lamperti1958} (or \cite[Proposition 2.3]{ChungLi2018}) on the classification of isometries of $\ell_p$ with $p\neq 2$, they prove that $\lproe(X)$ and $\lproe(Y)$ are isometrically isomorphic if and only if   $X$ and $Y$ are bijectively coarsely equivalent (\cite[Theorem~1.7]{ChungLi2018}). This is not yet known to be true in case $p=2$ without technical assumptions on the spaces. 

Our goal is to study mutual embeddings of $\ell_p$ uniform Roe algebras, in parallel with the work in the \cstar-setting carried on in~\cite{BragaFarahVignati2019}.
Differently from the isometric isomorphism  case studied in~\cite{ChungLi2018}, the geometric assumption present in the \cstar-case cannot be removed, as there are partial isometries on $\ell_p(\NN)$ which are not merely given by a permutation of $\NN$ together with a change of signs.  Hence, the assumption that all ghost idempotents in $\lproe(X)$ are compact will be necessary for our results.\footnote{For the definition of ghost operator, see~Definition~\ref{def:ghost}.} (This condition is fulfilled for instance if $X$ has Yu's property A, see Proposition \ref{PropPropAImpliesGhostsComp}.)

We now describe the main theorems of this paper and their corollaries. We would like to emphasis that all the results bellow holds for $p=2$ and are new even for $p=2$. The following should be compared with \cite[Theorem 1.2]{BragaFarahVignati2019}.

\begin{theorem}\label{thm:generalembedding}
Let $p\in [1,\infty)$, and $X$ and $Y$ be uniformly locally finite metric spaces. Assume that all ghost idempotents    in  $\lproe(Y)$  are compact. If  there exists an injective compact preserving bounded homomorphism  of $\lproe(X)$   into $\lproe(Y)$, then   there are a finite partition of $X$,  $X=\bigsqcup_{n=1}^k X_n$, and injective coarse maps $X_n\to Y$.
 \end{theorem}
 
The hypothesis on the map $\lproe(X)\to \lproe(Y)$ in Theorem \ref{thm:generalembedding} are fullfilled for instance if $\lproe(X)$ isomorphically  embeds into $\lproe(Y)$ by a  compact preserving map. Also, notice that one cannot hope to obtain in Theorem \ref{thm:generalembedding}  that $X$ coarsely embeds into $Y$. Indeed, it is easy to see that $\lproe(\Z)$ embeds into $\lproe(\N)$ by a compact preserving map, but $\Z$ does not coarsely embed into $\N$ (see \cite[Proposition 2.4]{BragaFarahVignati2019}). However the conclusion of Theorem \ref{thm:generalembedding} is enough to guarantee that if $Y$ has   finite asymptotic dimension,  finite decomposition complexity (FDC), \emph{or} property A, then $X$ has the same respective property  (\cite[Proposition 2.5,  Corollary 5.8 and Proof of Corollary 1.3(ii)]{BragaFarahVignati2019}). 

The next step is to deal with embeddings whose ranges are hereditary\footnote{If $A$ is a Banach algebra, $B\subseteq A$ is hereditary if $BAB\subseteq B$.}.

\begin{theorem}\label{ThmEmbHereditarySubAlg}
Let $p\in [1,\infty)$, and $X$ and $Y$ be uniformly locally finite metric spaces. Assume that all ghost idempotents in $\lproe(Y)$ are  compact. If there exists an isomorphic   embedding of $\lproe(X)$ onto a hereditary Banach subalgebra of $\lproe(Y)$, then $X$ coarsely embeds into $Y$.
\end{theorem}

Our results hold in a Banach algebraic setting, and   we do not require our maps to be isometric. We can therefore forget the structure of $^*$-algebra when dealing with uniform Roe \cstar-algebras, and also eliminate the condition requiring isometric isomorphisms in \cite[Theorem 1.7]{ChungLi2018}.

\begin{corollary}\label{cor:onlybanach}
Let $p\in[1,\infty)$, and $X$ and $Y$ be uniformly locally finite metric spaces. Assume that  all ghost idempotents  in $\lproe(X)$ and $\lproe(Y)$ are compact. If   $\lproe(X)$ and   $\lproe(Y)$ are isomorphic as Banach algebras, then $X$ is coarsely equivalent to  $Y$. 
\end{corollary}

Under the stronger assumption of property A, we can actually improve the results above and obtain injective coarse embeddings. 
\begin{theorem}\label{ThmEmbWithPropA}
Let $p\in[1,\infty)$,  $X$ and $Y$ be uniformly locally finite metric spaces,  and assume that $Y$ has property A. The following are equivalent.
\begin{enumerate}
\item\label{Item1ThmEmbWithPropA}  $X$  coarsely embeds into    $Y$ by an injective map.
\item \label{Item2ThmEmbWithPropA}  $\lproe(X)$ isometrically embeds as a Banach algebra onto a hereditary subalgebra of     $\lproe(Y)$.
\item \label{Item3ThmEmbWithPropA} $\lproe(X)$ embeds as a Banach algebra onto a hereditary subalgebra of     $\lproe(Y)$.
\end{enumerate}
\end{theorem}

Property A also allow us to obtain bijective coarse equivalences. 

\begin{theorem}\label{ThmIsomWithPropA}
Let $p\in[1,\infty)$,  $X$ and $Y$ be uniformly locally finite metric spaces,  and assume that $Y$ has property A. The following are equivalent.
\begin{enumerate}
\item\label{Item1ThmIsomWithPropA}      $X$ is bijectively coarsely equivalent to  $Y$.
\item \label{Item2ThmIsomWithPropA} $\lproe(X)$ and  $\lproe(Y)$ are isometrically isomorphic as   Banach algebras.
\item \label{Item3ThmIsomWithPropA} $\lproe(X)$ and  $\lproe(Y)$ are isomorphic as    Banach algebras.
\end{enumerate}
\end{theorem}

When $p=2$, we have the next corollary.

\begin{corollary}\label{CorIsoAsCstarIFFasBanAlg}
Let $X$ and $Y$ be uniformly locally finite metric spaces,  and assume that $Y$ has property A. Then  $\cstu(X)$ and $\cstu(Y)$ are isomorphic as Banach algebras if and only if they are isomorphic as \cstar-algebras. 
\end{corollary}

The techniques resembles the ones in \cite{BragaFarahVignati2019}. The lack of the operation $^*$ and  of isometric maps makes our proofs  more technical, with the payoff of extending some of the results in \cite{BragaFarah2018}, \cite{BragaFarahVignati2019} and \cite{ChungLi2018}.

\subsubsection*{Acknoledgments}
Part of this paper was written while BMB visited KU Leuven, and BMB is thankful for the hospitality of KU Leuven's department of mathematics. BMB is supported by the Simons Foundation and he would like to thank Ilijas Farah for useful discussions. AV is supported by a FWO Scholarship.

\section{Preliminaries}\label{section:preliminaries}
If $A$ is a Banach algebra, an element $P\in A$ is an \emph{idempotent} if $P^2=P$. Given idempotents $P,Q\in A$, we say that $Q$ is \emph{below} $P$, and write $Q\leq P$, if $PQ=QP=Q$. ($P$ and $Q$ will denote idempotents, while $p$ and $q$ will denote numbers in $[1,\infty)$). 

Let $\mathcal B(\ell_p(X))$ be the algebra of bounded linear operators on $\ell_p(X)$ for some set $X$. In this case, if $P\in\mathcal B(\ell_p(X))$ is a rank $1$ operator, there is a vector $\xi\in\ell_p(X)$ and a functional $f\in\ell_p(X)^*$ such that $P\eta=f(\eta)\xi$ for all $\eta\in \ell_p(X)$. We will use the notation $\xi\otimes f$ for such an operator, i.e., 
\[
(\xi\otimes f)(\eta)=f(\eta)\xi.
\]
If $x\in X$, $\delta_x\in\ell_p(X)$ denotes the vector given by the characteristic function on $X$. If $x,y\in X$, $e_{xy}\in\mathcal B(\ell_p(X))$ denotes the rank one operator given by
\[
e_{xy}\delta_z=\begin{cases}\delta_y&\text{ if }x=z\\
0&\text{ else.}
\end{cases}
\]
Let $(X,d)$ be a metric space and $n\in\NN$. An operator $T\in\mathcal B(\ell_p(X))$ has propagation $\leq n$, written $\propg(T)\leq n$, if 
\[
(T\delta_x)(y)\neq 0\Rightarrow d(x,y)\leq n.
\]
This is equivalent to requiring  that
\[
\norm{e_{xx}Te_{yy}}\neq 0\Rightarrow d(x,y)\leq n.
\]
\begin{definition}\label{def:lpunifroe}
Let $p\in[1,\infty)$ and $(X,d)$ be a metric space. The $\ell_p$ \emph{uniform Roe algebra of $X$}, denoted $B_u^p(X)$, is the subalgebra of $\mathcal B(\ell_p(X))$ given by the closure of the algebra of operators of finite propagation. When $p=2$ this is known as the \emph{uniform Roe \cstar-algebra of $X$}, denoted $\cstu(X)$.
\end{definition}

Let $(X_\lambda)_\lambda$ be an increasing net of finite subsets of $X$ such that $X=\bigcup_\lambda X_\lambda$. For each $\lambda\in\Lambda$, let $M^p(X_\lambda)\subseteq\mathcal B(\ell_p(X))$ be the algebra of $\C$-valued matrices $a\in \mathcal B(\ell_p(X))$ so that $e_{xx}ae_{yy}\neq 0$ implies $x,y\in X_\lambda$. Define 
\[
M_\infty^p(X)=\overline{\bigcup_\lambda M^p(X_\lambda)}.
\] 
Note that $M^p_\infty(X)\subseteq \lproe(X)$.

When $p\in (1,\infty)$, we have $M^p_\infty(X)=\cK(\ell_p(X))$, the latter being the ideal of compact operators on $\ell_p(X)$. If $p=1$, then  $M_\infty^p(X)\subseteq\cK(\ell_p(X))$, but the inclusion is proper: fix $x_0\in X$ and consider the operator $\delta_{x_0}\otimes f$ where $f\in\ell_1(X)^*=\ell_\infty(X)$ is the vector which is $1$ at each entry. Then $\delta_{x_0}\otimes f$ has rank $1$ and therefore belongs to $\mathcal K(\ell_1(X))$, but it has distance $1$ from each $M^p(X_\lambda)$. More than that: if $X$ is an unbounded metric space, $\delta_{x_0}\otimes f$ has infinite propagation, as $\norm{e_{x_0x_0}(\delta_{x_0}\otimes f)e_{yy}}=1$ for all $y\in X$, so $\mathcal K(\ell_1(X))\nsubseteq \loneroe(X)$.
 One can easily check that
\[M^p_\infty(X)=\lproe(X)\cap\cK(\ell_2(X))\]
for all $p\in [1,\infty)$. Therefore, a  map $\Phi\colon\lproe(X)\to\lproe(Y)$ is compact preserving if and only if $\Phi[M_\infty^p(X)]\subseteq M_\infty^p(Y)$.  

A metric space $(X,d)$ is said to be \emph{uniformly locally finite}, abbreviated \emph{u.l.f.}, if 
\[
\sup_{x\in X}|\{y\mid d(x,y)\leq n\}|<\infty,\,\,\text{ for all }n\in\NN.
\]
\begin{definition}\label{def:ghost}
Let $p\in[1,\infty)$ and  $X$ be a u.l.f. metric space. An operator $a\in \lproe(X)$ is called a \emph{ghost} if for all $\eps>0$ there exists a finite $A\subset X$ such that $\|e_{yy}ae_{xx}\|<\eps$, for all $x,y\in X\setminus A$.
\end{definition}

The algebra of ghosts always contains $M_\infty^p(X)$. On the other hand, if a ghost has finite propagation, then it easily follows that it must belong to $M_\infty^p(X)$. Ghosts operators are closely related to Yu's property A (\cite{Yu2000}). Instead of introducing the original definition of property A, we use a characterization due to J. Roe and R. Willett (see \cite[Theorem 1.3]{RoeWillett2014}).

\begin{definition}
A u.l.f. metric space $X$ has \emph{property A} if all ghost operators in $\cstu(X)$ are compact. 
\end{definition}

The next proposition is essentially \cite[Proposition 11.43]{RoeBook}.

\begin{proposition}\label{PropPropAImpliesGhostsComp}
Let $X$ be a u.l.f. metric space with property A and $p\in [1,\infty)$. Then all ghosts in $\lproe(X)$ belong to $M_\infty^p(X)$.
\end{proposition}

\begin{proof}
If  $p>1$, \cite[Corollary 6.5]{SpakulaWillett2017} gives  a sequence $(\mathcal M_n\colon\lproe(X)\to\lproe(X))_n$ of norm one operators so that $\mathcal M_n(a)$ has finite propagation and  $a=\lim_n\mathcal M_n(a)$ for all $a\in \lproe(X)$. It follows straightforwardly from the definition of $(\mathcal M_n)_n$ that if $a\in \lproe(X)$ is a ghost, so is $\mathcal M_n(a)$. If $p=1$, a careful look at \cite[Lemma 6.3, Lemma 6.4, and Corollary 6.5]{SpakulaWillett2017} give us the same sequence $(\mathcal M_n\colon\lproe(X)\to\lproe(X))_n$. Indeed, the interested reader only need to perform the following modification in those proofs: given  a map $\varphi\colon X\to [0,1]$ and $p\in (1,\infty)$, the authors of \cite{SpakulaWillett2017} define $\varphi^{p/q}\colon X\to [0,1]$ by letting $\varphi^{p/q}(x)=(\varphi(x))^{p/q}$ for all $x\in X$, where $1/p+1/q=1$. If $p=1$, we only need to make the natural modification, i.e., define $\varphi^{1/\infty}\colon X\to [0,1]$ by letting 
\[\varphi^{1/\infty}(x)=\left\{\begin{array}{ll}
0,& \text{ if } \varphi(x)=0,\\
1,& \text{ otherwise.}
\end{array}\right.\]
Running the proofs of  \cite[Lemma 6.3, Lemma 6.4, and Corollary 6.5]{SpakulaWillett2017} with this modification gives us the desired result.

Fix a ghost $a\in \lproe(X)$; each $\mathcal M_n(a)$ is a ghost of finite propagation, so it belongs to $M_\infty^p(X)$.   Since $a=\lim_n\mathcal M_n(a)$, $a\in M_\infty^p(X)$. 
\end{proof}

When considering maps $\lproe(X)\to\lproe(Y)$, for u.l.f. spaces $X$ and $Y$, our running assumption is that all ghost idempotents belong to $M_\infty^p(Y)$. This automatically implies that $Y$ is u.l.f. (cf. \cite[Lemma 2.2]{BragaFarahVignati2018}).

\subsection{Coarse-like properties}
If $(X,d)$ and $(Y,\partial)$ are metric spaces, a function $f\colon X\to Y$ is said \emph{coarse} if for all $r>0$ there is $s>0$ such that 
\[
d(x,y)<r\text{ implies }\partial(f(x),f(y))<s
\] 
and \emph{expanding} if for all $s>0$ there is $r>0$ such that 
\[
d(x,y)>r\text{ implies }\partial(f(x),f(y))>s.
\]
A map which is both coarse and expanding  is said to be a \emph{coarse embedding}. Two functions $f,g\colon X\to Y$ are \emph{close} if 
\[
\sup_{x\in X}\partial(f(x),g(x))<\infty,
\] and two spaces $X$ and $Y$ are \emph{coarsely equivalent} if there are two coarse maps  $f\colon X\to Y$ and $g\colon Y\to X$ such that  $g\circ f$ and $f\circ g$ are close to $\mathrm{Id}_X$ and $\mathrm{Id}_Y$ respectively. (This automatically implies that $f$ and $g$ are expanding).

We want to study how maps between $\ell_p$ uniform Roe algebras can encode the geometry of the spaces involved. The following notion of regularity was introduced in \cite{BragaFarah2018} and later formalized in \cite{BragaFarahVignati2018} in case $p=2$.

\begin{definition}
Let $p\in [1,\infty)$, and  $X$ and $Y$ be  u.l.f.  metric spaces. 
\begin{enumerate}
\item Given $\eps>0$ and $r\geq 0$, an operator $a\in \cB(\ell_p(X))$ is called \emph{$\eps$-$r$-approximable} if there exists $b\in \cB(\ell_p(X))$ with $\propg(b)\leq r$ so that $\|a-b\|\leq \eps$.
\item A map $\Phi\colon\lproe(X)\to\lproe(Y)$ is \emph{coarse-like} if for all $\eps>0$ and all $r>0$ there exists $s>0$ so that $\Phi(a)$ is $\eps$-$s$-approximable for all $a\in \lproe(X)$ with $\|a\|\leq 1$ and $\propg(a)\leq r$.
\end{enumerate}
\end{definition}

Being able to work with coarse-like maps is a technical key in many of our arguments. Examples of coarse-like maps are strongly continuous maps\footnote{A strongly continuous nonzero homomorphism $\Phi\colon\lproe(X)\to\lproe(Y)$ must be necessarily injective, as $M_\infty^p(X)$ is a strongly dense subalgebra of $\lproe(X)$ contained in all ideals. We will use this fact without mentioning it.}. The next proposition  was proved as \cite[Proposition 3.3]{BragaFarahVignati2019} for $p=2$, and it is a simple consequence of \cite[Lemma 4.9]{BragaFarah2018}. If $p\neq 2$, the proof is essentially the same, and we leave it to the reader.

\begin{proposition}\label{PropCoarseLike}
Let $p\in [1,\infty)$, and  $X$ and $Y$ be u.l.f. metric spaces. Every compact preserving strongly continuous linear map $\Phi\colon\lproe(X)\to \lproe(Y)$ is coarse-like.\qed
\end{proposition} 
\subsubsection*{Quasi-locality}
In \cite{SpakulaZhang2018} it was introduced a second notion of regularity capable, in certain cases, of detecting whether an operator belongs to $\lproe(X)$.

\begin{definition}\label{DefQuasiLocal}
Let $p\in [1,\infty)$ and $(X,d)$ be a metric space.
\begin{enumerate}
\item Let $\eps>0$ and $s\in\N$. An operator $a\in \cB(\ell_p(X))$ is \emph{$\eps$-$s$-quasi-local} if $\|\chi_Aa\chi_B\|\leq\eps$ for all $A,B\subseteq X$ with $d(A,B)>s$.  
\item An operator $a\in \cB(\ell_p(X))$ is \emph{quasi-local} if for all $\eps>0$ there exists $s>0$ such that $a$ is $\eps$-$s$-quasi-local.
\end{enumerate}
\end{definition}
If $X$ is u.l.f. and if either  $p=1$ or $X$ has property $A$, then all quasi-local operators in $\cB(\ell_p(X))$ belong to $\lproe(X)$, see \cite[Theorem 3.3]{SpakulaZhang2018}. It is not known whether this holds for all $p\in[1,\infty)$ and all u.l.f. spaces.
The proof of the following is analogous to that of \cite[Lemma 3.8]{BragaFarahVignati2019}.
\begin{lemma}\label{LemmaSumOfQuasiLocalSOT}
Let $p\in [1,\infty)$, $X$ be a metric space, and $\{a_n\}_{n}\subseteq\cB(\ell_p(X))$ be quasi-local operators such that $\sum_{n\in M}a_n$ converges is the strong operator topology to a quasi-local element in $\cB(\ell_p(X))$ for all $M\subseteq \N$. Then for every  $\eps>0$ there exists $n\in\N$ such that   $a_n$ is  $\eps$-$n$-quasi-local.\qed
\end{lemma}

\section{Embeddings}\label{section:embeddings} 
The first part of this section is dedicated to the study of strongly continuous maps between $\ell_p$ uniform Roe algebras. Later, we show that under some reasonable conditions we can assume that the map in Theorem~\ref{thm:generalembedding} is strongly continuous. In \S\ref{subsec:her}  we study maps whose ranges are hereditary and in \S\ref{subsec:propA} we prove stronger results in presence of property A.

Compare the following with \cite[Lemma 5.1]{BragaFarahVignati2019}.
\begin{lemma}\label{LemmaPickMapf}
Let $p\in [1,\infty)$, and  $X$ and $Y$ be u.l.f. metric spaces. Let $\Phi\colon \lproe(X)\to\lproe(Y)$ be a strongly continuous compact preserving nonzero homomorphism. Assume that all  ghosts idempotents in $\lproe(Y)$ belong to $M_\infty^p(Y)$. Then 
\[\delta=\inf_{x\in X}\sup_{y,z\in Y}\|e_{zz}\Phi(e_{xx})e_{yy}\|>0.\]
\end{lemma}

\begin{proof}
By contradiction, suppose that there exists a sequence $(x_n)_n\subseteq X$ such that $\|e_{zz}\Phi(e_{x_nx_n})e_{yy}\|<2^{-n}$ for all $n\in\N$ and all $y,z\in X$. Without loss of generality,  we can assume the elements of $(x_n)_n$ to be distinct. Set   
\[
P=\Phi\Big(\sum_ne_{x_nx_n}\Big)=\mathrm{SOT}\text{-}\lim_k\sum_{n\leq k}\Phi(e_{x_nx_n}),
\]
 where the second equality follows from the fact that $\Phi$ is strongly continuous. Note that $P$ is an idempotent of infinite rank, so $P\not\in M_\infty^p(Y)$. Let $\eps>0$ and pick $N\in \N$ with $2^{N}<\eps/2$. Since $\Phi$ is compact preserving, $\Phi(e_{x_nx_n})$ is a ghost for all $n\in\N$. Pick a finite $A\subset Y$ such that $\|e_{zz}\Phi(e_{x_nx_n})e_{yy}\|<\eps/(2N)$ for all $y,z\not\in A$ and all $n\leq N$. Then
\[\|e_{zz}Pe_{yy}\|\leq\sum_{n=1}^N \|e_{zz}\Phi(e_{x_nx_n})e_{yy}\|+\sum_{n>N} \|e_{zz}\Phi(e_{x_nx_n})e_{yy}\|\leq \eps\]
for all $y,z\not\in A$, that is, $P$ is a ghost idempotent not belonging to $M_\infty^p(Y)$, a contradiction.
\end{proof}

\begin{lemma}\label{LemmaCloseAndUnfFiniteToOne}
Let $p\in [1,\infty)$, $\delta>0$,  $X$ and $Y$ be  u.l.f. metric spaces, $\Phi\colon \lproe(X)\to\lproe(Y)$ be a bounded homomorphism, and $f,g\colon X\to Y$  be maps. Then:
\begin{enumerate}
\item\label{ItemCloseAndUnfFiniteToOne1} if $\Phi$ is coarse-like and $\|e_{g(x)g(x)}\Phi(e_{xx})e_{f(x)f(x)}\|\geq \delta$ for all $x\in X$, then $f$ and $g$ are close;
\item\label{ItemCloseAndUnfFiniteToOne2} if $\|\Phi(e_{xx})e_{f(x)f(x)}\|\geq \delta$ for all $x\in X$, then $f$ is uniformly finite-to-one\footnote{A function $f\colon X\to Y$ is said \emph{uniformly finite-to-one} if $\sup_{y\in Y}|f^{-1}(\{y\})|<\infty$.}.
\end{enumerate}
\end{lemma}

\begin{proof}
\eqref{ItemCloseAndUnfFiniteToOne1} Since $\Phi$ is coarse-like,   there exists $m>0$ such that $\Phi(e_{xx})$ is $\delta/2$-$m$-approximable for all $x\in X$. In particular, $\partial(y,y')>m$ implies $\|e_{yy}\Phi(e_{xx})e_{y'y'}\|<\delta$. So, $\partial(f(x),g(x))\leq m$ for all $x\in X$. 

\eqref{ItemCloseAndUnfFiniteToOne2} Since $\ell_p(X)$ has cotype\footnote{Let $p\in [2,\infty)$. A Banach space $\mathcal X$ is said to have \emph{cotype $p$ } if there exists a constant $c>0$ such that 
\[c\sum_{i=1}^n\|x_i\|^p\leq \frac{1}{2^n}\sum_{(\eps_i)_{i=1}^{n}}\Big\|\sum_{i=1}^n\eps_ix_i\Big\|^p\]
for all $n\in\N$ and all $x_1,\ldots, x_n\in \mathcal X$.} ${\max\{2,p\}}$ (see \cite[Theorem 6.2.14]{AlbiacKalton2006}), there exists $c>0$ such that 
\[\frac{1}{2^n}  \sum_{(\eps_i)_i\in \{-1,1\}^n}\Big\|\sum_{i=1}^n\eps_i\xi_i\Big\|^{\max\{2,p\}}  \geq c \sum_{i=1}^n\|\xi_i\|^{\max\{2,p\}} \]
for all $n\in\N$ and all $\xi_1,\ldots, \xi_n\in \ell_p(X)$.
Let $f\colon X\to Y$ be a map such that $\|\Phi(e_{xx})\delta_{f(x)}\|\geq \delta$ for all $x\in X$. Let $y\in Y$ and $F\subseteq X$ be a  finite subset such that $f(x)=y$ for all $x\in F$. Since 
\[
\Big\|\sum_{x\in F}\eps_x\Phi(e_{xx})\Big\|=\Big\|\Phi\Big(\sum_{x\in F}\eps_xe_{xx}\Big)\Big\|\leq \|\Phi\|
\] for all $(\eps_x)_{x\in F}\in \{-1,1\}^{F}$,   it follows that 
\[
\|\Phi\|^{\max\{2,p\}}\geq\frac{1}{2^{|F|}} \sum_{(\eps_i)_i\in \{-1,1\}^{F}}\Big\|\sum_{x\in F}\eps_x\Phi(e_{xx})\delta_y\Big\|^{\max\{2,p\}} \geq c\delta^{\max\{2,p\}}|F|.
\]
This    implies that 
\[
|F|\leq c^{-1}\delta^{-{\max\{2,p\}}}\|\Phi\|^{\max\{2,p\}}.\qedhere
\]
\end{proof}

We just showed that strongly continuous maps are capable of relating the geometry of $X$ to the one of $Y$. We will work to be able to assume that our maps are strongly continuous. 
Our next result plays the role of \cite[Lemma 3.4]{BragaFarahVignati2018}; lacking adjoints,  we have to slightly modify its proof.
\begin{lemma} \label{L:AA} Let $p\in [1,\infty)$,   $X$ and $Y$ be 
metric spaces 
 and  $\Phi\colon \ell_\infty(X)\to \lproe(Y)$ be 
  a strongly continuous  homomorphism.  
Then for every   $b\in M_\infty^p(Y)$  and 
every $\e>0$ there exists a finite  $F\subseteq X$ such that 
for all contractions $a\in \ell_\infty(X\setminus F)$ we have 
$\max(\norm{b\Phi(a)},\norm{\Phi(a)b})<\e$. 
\end{lemma}

\begin{proof} Suppose  $b\in M_\infty^p(Y)$ and $\eps>0$  contradict  the thesis. By approximating $b$, we can assume that $b\in M^p(F)$ for some finite $F\subseteq Y$. The strong continuity of $\Phi$ allows us to   pick a sequence  $(E_n)_n$ of  disjoint finite subsets of $X$ and contractions $a_n\in \ell_\infty(E_n)$ such that $\|b\Phi(a_n)\|\geq \eps/2$ for all $n\in\N$ (if we can only assume that $\|\Phi(a_n)b\|\geq \delta/2$ for all $n\in\N$, the proof will follow analogously). Then, for all $n\in\N$, 
\[
\|\chi_F\Phi(a_n)\|\geq\|\chi_Fb\chi_F\Phi(a_n)\|=\|b\Phi(a_n)\|\geq\eps.
\]
\begin{claim}\label{ClaimRows}
Replacing  $F$ with a larger finite subset of $Y$,  we can assume that $\inf_n\|\chi_F\Phi(a_n)\chi_F\|>0$.
\end{claim}

\begin{proof} Since $F$ is finite, $\Phi(a_n)\in \lproe(Y)$ and $Y$ is u.l.f., each $\chi_F\Phi(a_n)$ can be arbitrarily well approximated by some $\chi_F\Phi(a_n)\chi_G$ for a sufficiently large finite $G\subset Y$. Hence, if the claim fails,    we can pick  a sequence $(F_n)_n$ of finite disjoint subsets of $Y$ so that $\delta=\inf_n\|\chi_F\Phi(a_n)\chi_{F_n}\|>0$ and $\|\chi_F\Phi(a_n)\chi_{F_m}\|<2^{-n-1}\delta$ for all $n\neq m$. Similarly, there is a finite $G$ such that $\norm{\chi_F\Phi(\sum_na_n)-\chi_F\Phi(\sum_na_n)\chi_{G}}<\delta/4$.  Pick $m$ such that $G\cap F_m=\emptyset$. The contradiction then comes from 
\[
\frac{\delta}{4}> \Big\|\chi_{F}\Phi\Big(\sum_na_n\Big)\chi_{F_m}\Big\|\geq\|\chi_{F}\Phi(a_m)\chi_{F_m}\| -\sum_{n\neq m}\|\chi_{F}\Phi(a_n)\chi_{F_m}\|\geq \frac{\delta}{2}.\qedhere
\]
\end{proof}
 
Suppose $\inf_n\|\chi_F\Phi(a_n)\chi_F\|>0$. Since $\sum_na_n$ converges in the strong operator topology and $\Phi$ is strongly continuous, $\sum_n\Phi(a_n)$ converges in the strong operator topology. Since $F$ is finite, the sum $\sum_{n}\chi_F\Phi(a_n)\chi_F$ converges in norm; contradiction, since  $\inf_n\norm{\chi_F\Phi(a_n)\chi_F}>0$.
\end{proof}

 Let $p\in[1,\infty)$ and   $X$ be a countable metric space. Since   $M_\infty^p(X)=\cK(\ell_p(X))\cap \lproe(X)$ is an ideal of $\cB(\ell_p(X))$, we define the \emph{$\ell_p$ uniform Roe corona of $X$} by 
 \[
 \mathrm{Q}^p_u(X)=\lproe(X)/M_\infty^p(X) \] 
 and let $\pi_X\colon \lproe(X)\to\mathrm{Q}^p_u(X)$ be the quotient map\footnote{We refer the reader to \cite{BragaFarahVignati2018} for a detailed study of uniform Roe \cstar-coronas.}. Given a metric space $Y$, maps $\Phi\colon\ell_\infty(X)\to \cB(\ell_p(Y))$ and  $\Lambda\colon\ell_\infty(X)/c_0(X)\to \mathrm{Q}^p_u(Y)$, and an ideal $\SJ\subset\cP(X)$, we say that \emph{$\Phi$ lifts $\Lambda $ on $\SJ$} if  
 \[
\Lambda\circ\pi_X(\chi_A)=\pi_Y\circ\Phi(\chi_A), \,\, 
\text{ for all }A\in \SJ.
\]
 The proof of the following is analogous to the one of \cite[Proposition 3.3]{BragaFarahVignati2018}, by using Lemma \ref{L:AA} instead of  \cite[Lemma 3.4]{BragaFarahVignati2018}.

\begin{proposition}\label{prop:stayinside}
Let $p\in [1,\infty)$, and  $X$ and $Y$ be countable metric spaces. 
 Suppose that $\Phi\colon\ell_\infty(X)\to\cB(\ell_p(Y))$ is a strongly continuous
 homomorphism  lifting a  homomorphism  $\Lambda\colon\mathrm{Q}^p_u(X)\to \mathrm{Q}^p_u(Y)$  on a nonmeager ideal  $\SI\subset \cP(X)$  
containing all finite subsets of $X$. Then $\Phi$ is coarse-like and the 
image of~$\Phi$ is contained in $\lproe(Y)$.\qed
\end{proposition}

The next result shows that we can safely assume that the map in Theorem~\ref{thm:generalembedding} is strongly continuous; compare it to \cite[Theorem 4.3]{BragaFarahVignati2019}.

\begin{theorem}\label{ThmEmbImpliesStrongEmb}
Let $p\in [1,\infty)$,    $X$ and $Y$ be  u.l.f. metric spaces,  and  $\Phi\colon \lproe(X)\to \lproe(Y)$ be an injective bounded homomorphism. Assume that one of the following holds:
\begin{enumerate}
\item \label{ItemThmEmbImpliesStrongEmbP1} $p=1$,
\item \label{ItemThmEmbImpliesStrongEmbProA} $Y$ has property A, or
\item\label{ItemThmEmbImpliesStrongEmbCompPres} $\Phi$ is compact preserving.
\end{enumerate} 
 Then, there exists an idempotent $P\in \lproe(Y)$  in the commutator of $\Phi(\lproe(X))$ such that  
 \[
 a\in \lproe(X)\mapsto P\Phi(a)P\in \lproe(Y)
 \] 
is a strongly continuous nonzero homomorphism.
\end{theorem}

\begin{proof}
 Let $(X_n)_n$ be an increasing sequence of finite subsets of $X$ with $\bigcup_nX_n=X$.
\begin{claim} 
The sequence $(\Phi(\chi_{X_n}))_n$ converges in the strong operator topology to an idempotent of norm at most $\|\Phi\|$.
\end{claim}
 
\begin{proof}
We only show that $(\Phi(\chi_{X_n}))_n$ converges in the strong operator topology, the other claims are trivial. For that, it is enough to show that $(\Phi(\chi_{X_n})\xi)_n$ is Cauchy for all $\xi\in\ell_p(Y)$. Suppose this fails. Then there exist $\eps>0$, $\xi\in \ell_p(Y)$ and a sequence $(E_n)_n$ of disjoint finite subsets of $X$ so that $\|\Phi(\chi_{E_n})\xi\|\geq \eps$ for all $n\in\N$. Since $\ell_p(Y)$ has cotype $\max\{2,p\}$,  proceeding as in Lemma \ref{LemmaCloseAndUnfFiniteToOne}, there exists $c>0$ so that 
\begin{align*}
\|\Phi\|^{\max\{2,p\}}&\geq \frac{1}{2^n}\sum_{\eps_i\in \{-1,1\}^n}\Big\|\sum_{i=1}^n\eps_i\Phi(\chi_{E_i})\xi\Big\|^{\max\{2,p\}}\\
&\geq c\sum_{i=1}^n\|\Phi(\chi_{E_i})\xi\|^{\max\{2,p\}}\\
&\geq c\eps^{\max\{2,p\}}n
\end{align*}
for all $n\in\N$; contradiction.
\end{proof}

Let  $P=\mathrm{SOT}\text{-}\lim_n\Phi(\chi_{X_n})$. Note that  $P$ commutes with $\Phi[M^p(X_k)]$ for all $k$. As $\bigcup_nM^p(X_k)$ is dense in $M_\infty^p$, $P$ commutes with $\Phi[M_\infty^p]$. More than that, $P$ commutes with $\Phi[\lproe(X)]$: let $a\in \lproe(X)$. Then  $\Phi(a\chi_{X_n})P=P\Phi(a\chi_{X_n})$ for all $n\in\N$, and  it follows that 
\[
\Phi(a)P=\Phi(a)P^2=\mathrm{SOT}\text{-}\lim_n\Phi(a\chi_{X_n})P=\mathrm{SOT}\text{-}\lim_nP\Phi (a\chi_{X_n})=P\Phi(a)P.
\]
Analogously $P\Phi(a)P=\Phi(a)P$. In particular, the map $\Phi^P=P\Phi P$ is a homomorphism, and, since $\Phi$ is injective, we have that $P\neq 0$ and so $\Phi^P\neq 0$.

\begin{claim}
$\Phi^P$ is strongly continuous.
\end{claim}

\begin{proof}
Let $a\in \lproe(X)$  and $(a_k)_k$ be a sequence in $\lproe(X)$ with $a=\mathrm{SOT}\text{-}\lim_ka_k$. Without loss of generality,   assume that $a$ and each $a_k$ are contractions. Let $H=\mathrm{Im}(P)$ and $H'=\mathrm{Im}(1-P)$. Since $\ell_p(Y)=H\oplus H'$ and $\Phi^P(\lproe(X))[H']=0$, we only need to show that 
\[
\Phi^P(a)\xi=\lim_k\Phi^P(a_k)\xi,
\] for all   $\xi\in H$. Fix $\xi\in H$ and  $\eps>0$. Pick a large enough $\ell\in \N$   so that $\|\Phi(\chi_{X_\ell})\xi-\xi\|<\eps$.
 Since $X_\ell$ is finite, $a\chi_{X_\ell}=\lim_ka_k\chi_{X_\ell}$, which implies that $\Phi^p(a \chi_{X_\ell})=\lim_k\Phi^p(a_k\chi_{X_\ell})$. Since $\Phi$ and $\Phi^p$ agree on $M_\infty^p$, we have that  
\[
\limsup_k\|\Phi^p(a_k)\xi-\Phi^p(a)\xi\|\leq \lim_k\|\Phi^p(a_k \chi_{X_\ell})\xi-\Phi^p(a \chi_{X_\ell})\xi\|+2\eps=2\eps.\qedhere
\]
\end{proof}

Since $\Phi^P\neq 0$, it follows from its strong continuity that $\Phi^P$ is injective. We are left to show that $P\in\lproe(Y)$ under the assumptions in the theorem.

\eqref{ItemThmEmbImpliesStrongEmbP1} and \eqref{ItemThmEmbImpliesStrongEmbProA}: Suppose $P=\Phi^P(1)\not\in \lproe(Y)$. In either $p=1$ or $Y$ has property A,  \cite[Theorem 3.3]{SpakulaZhang2018} implies that there exists $\eps>0$ so that $P$ is not $2\eps$-$s$-quasi-local for all $s>0$. The proof of the next claim goes exactly as the one of \cite[Claim 4.7]{BragaFarahVignati2019}.

\begin{claim}\label{ClaimClaim1}
For all finite  $E_0\subseteq X$ and all $s>0$ there exists a finite $E\subseteq X\setminus E_0$ such that $\Phi^P(\chi_{ E})$ is not $\eps$-$s$-quasi-local.\qed
\end{claim}

By the previous claim,  pick a disjoint sequence $(E_n)_n$ of finite subset of $X$ such that $(\Phi^p(\chi_{E_n}))_n$ is not $\eps$-$n$-quasi-local for all $n\in\N$. Since $\Phi^P$ is strongly continuous, this contradicts the conclusion of  Lemma \ref{LemmaSumOfQuasiLocalSOT}.

\eqref{ItemThmEmbImpliesStrongEmbCompPres}: Since $\Phi[M_\infty^p(X)]\subseteq M_\infty^p(Y)$,  $\Phi$ induces a  homomorphism 
\[
\tilde\Phi\colon \mathrm{Q}^p_u(X)\to \mathrm{Q}^p_u(Y).
\]
 Let $\SJ=\{A\subseteq X\mid \Phi(\chi_A)=\Phi^{P}(\chi_A)\}$, so $\Phi^P\restriction \ell_\infty(X)$ lifts $\tilde \Phi\restriction \ell_\infty(X)/c_0(X)$ on $\SJ$. Proceeding as in  \cite[Claim 4.6]{BragaFarahVignati2019}, we have that 
is a ccc/Fin ideal\footnote{An ideal $\SI\subseteq\mathcal P(X)$ is ccc/Fin if every family  $(A_i)_{i\in I}$ of subsets of $X$ so that $|A_i\cap A_j|<\infty$ and $A_i\not\in \SI$ for all $i\in I$ is countable.}, and in particular is nonmeager and contains all finite subsets of $X$. By \cite[Proposition 3.3]{BragaFarahVignati2018},   $\Phi^P\restriction \ell_\infty(X)$ is then coarse-like and the image of $\Phi^P$  is contained in $\lproe(Y)$. In particular,  $P=\Phi^P(1)\in \lproe(Y)$.
\end{proof}

We can finally present the proof of Theorem \ref{thm:generalembedding} (cf. \cite[Theorem 1.2(2)]{BragaFarahVignati2019}).

\begin{proof}[Proof of Theorem \ref{thm:generalembedding}]
Let $\Phi\colon \lproe(X)\to \lproe(Y)$ be an injective   compact preserving bounded homomorphism. By Theorem \ref{ThmEmbImpliesStrongEmb}, we can also assume that $\Phi$ is strongly continuous. Let $\delta>0$, and  $f\colon X\to Y$ and $g\colon X\to Y$ be given by Lemma \ref{LemmaPickMapf}, i.e., 
\[\delta=\inf_{x\in X}\|e_{g(x)g(x)}\Phi(e_{xx})e_{f(x)f(x)}\|>0.\]
 By  Lemma \ref{LemmaCloseAndUnfFiniteToOne}, $f$ and $g$ are close to each other and uniformly finite-to-one. Let $m$ be such that $\partial (f(x),g(x))\leq m$ for all $x\in X$.

Fix  $x_0\in X$. Since $\Phi$ is  compact preserving, $\text{Im}(\Phi(e_{x_0x_0}))$ is finite dimensional. By our choice of $\delta $ and $f$, $\|\Phi(e_{xx})\delta_{f(x)}\|\geq \delta$ for all $n\in\N$. Notice that $\Phi(e_{xx_0})\restriction\text{Im}(\Phi(e_{xx}))$ is an isomorphism onto $\text{Im}(\Phi(e_{x_0x_0}))$, with inverse $\Phi(e_{x_0x})\restriction\text{Im}(\Phi(e_{x_0x_0}))$. In particular,
\[
\|\Phi\|^{-1}\cdot \|\Phi(e_{xx})\delta_{f(x)}\|\leq \|\Phi(e_{xx_0})\delta_{f(x)}\|\leq \|\Phi\| 
\]
for all $x\in X$. Hence
\[\Phi(e_{xx_0})\delta_{f(x)}\in D=\big\{\xi\in \ell_p(Y)\mid  \|\xi\|\in [\|\Phi\|^{-1}\delta,\|\Phi\|]\big\}\cap \text{Im}(\Phi(e_{x_0x_0}))\]
for all $x\in X$. Since $D$ is compact there is a finite partition $D=\bigsqcup_{n=1}^kD_n$ such that $\diam(D_n)<\delta/(2\|\Phi\|)$ for all $n\leq k$. Let
\[
X_n=\{x\in X\mid \Phi(e_{xx_0})\delta_{f(x)}\in D_n\}.
\]
Clearly, $X=\bigsqcup_{n=1}^k X_n$.  If $x_1,x_2\in X_n$, it follows that
\begin{align*}
\|e_{g(x_2)g(x_2)}&\Phi(e_{x_1x_2})e_{f(x_1)f(x_1)}-e_{g(x_2)g(x_2)}\Phi(e_{x_2x_2})e_{f(x_2)f(x_2)}\|\\
&\leq \|\Phi(e_{x_1x_2})e_{f(x_1)f(x_1)}-\Phi(e_{x_2x_2})e_{f(x_2)f(x_2)}\|\\
&=\|\Phi(e_{x_0x_2})\Phi(e_{x_1x_0})e_{f(x_1)f(x_1)}-\Phi(e_{x_0x_2})\Phi(e_{x_2x_0})e_{f(x_2)f(x_2)}\|\\
&\leq \|\Phi\|\cdot\|\Phi(e_{x_1x_0})\delta_{f(x_1)}-\Phi(e_{x_2x_0})\delta_{f(x_2)}\|\leq\delta/2.
\end{align*}
By our choice of $f$ and $g$, this implies that  
\begin{equation}\label{Equation1}
\|e_{g(x_2)g(x_2)}\Phi(e_{x_1x_2})e_{f(x_1)f(x_1)}\|\geq \delta/2.
\end{equation}

\begin{claim}
For each $n\leq k$, the map $f\restriction X_n$ is coarse.
\end{claim}

\begin{proof}
Let $r>0$. By Proposition \ref{PropCoarseLike}, there exists $s>0$ such that $\Phi(e_{x_1x_2})$ is $\delta/4$-$s$-approximated for all $x_1,x_2\in X$ with $d(x_1,x_2)<r$. In particular, for all $x_1,x_2\in X$ with $d(x_1,x_2)<r$, it follows that
\begin{equation}\label{TagAst2}
\|e_{y_2y_2}\Phi(e_{x_1x_2})e_{y_1y_1}\|\leq\frac{\delta}{4}\tag{$*$}
\end{equation}
 for all $y_1,y_2\in Y$ with $\partial (y_1,y_2)>s$.  Hence, \eqref{Equation1} implies that \[\partial (f(x_1),f(x_2))\leq \partial(f(x_1),g(x_2))+\partial(g(x_2),f(x_2))\leq s+m\] for all $x_1,x_2\in X_n$ with $d(x_1,x_2)<r$.
\end{proof}
Since $f$ is uniformly finite-to-one, by splitting each $X_n$ into finite many pieces if necessary, we can assume that $f\restriction X_n$ is injective.
\end{proof}

\subsection{Maps onto hereditary subalgebras}\label{subsec:her}
Recall that if $B\subseteq A$ are Banach algebras, we say that $B$ is \emph{hereditary} if  $BAB\subseteq B$.

\begin{lemma}\label{LemmaPhiStronglyContAndU}\label{We need an embedding here}
Let $p\in[1,\infty)$, and $X$ and $Y$ be  u.l.f. metric spaces. Let $\Phi\colon \lproe(X)\to\lproe(Y)$ be an   embedding onto a hereditary Banach subalgebra of $\lproe(Y)$. Then  there exists a surjective  isomorphism $U\colon \ell_p(X)\to  \mathrm{Im}(\Phi(1))$  with $\norm{\Phi}=\norm{U}$ and such that 
\[
\Phi(a)=UaU^{-1}\Phi(1)
\]
for all $a\in \lproe(X)$. $\Phi$ is then strongly continuous and compact preserving. 
\end{lemma}

\begin{proof}
First, note that $\Phi(e_{xx})$ has rank 1 for all $x\in X$. If not, since $\Phi(e_{xx})\in \lproe(Y)$, there exists a nontrivial idempotent  $Q\in \lproe(Y)$   strictly below $\Phi(e_{xx})$, for some $x\in X$. Indeed, one can take for instance $Q=\xi\otimes f$, where $\xi$ is a unit vector in $\mathrm{Im}(\Phi(e_{xx}))$ and $f=g\circ\Phi(e_{xx})$ for some $g\in \ell_p(Y)^*$ with finite support so that $g(\xi)=1$. Since $g$ has finite support and $\Phi(e_{xx})$ belongs to $\lproe(Y)$, it easily follows that $Q\in \lproe(Y)$. Since the image of $\Phi$ is hereditary, there is an idempotent $R\neq e_{xx}$ such that $\Phi(R)=Q$, hence $Re_{xx}=e_{xx}R=R$, a contradiction.

Fix $x_0\in X$, and pick a unit vector $\zeta\in\ell_p(Y)$ and $f\in \ell_p(Y)^*$ such that $\Phi(e_{x_0x_0})=\zeta\otimes f$ and $\|f\|=\|\Phi\|$.   Define $U\colon\ell_p(X)\to\ell_p(Y)$ by 
\[
U\xi=\Phi(\xi\otimes \delta_{x_0})\zeta\text{ for all }\xi\in \ell_p(X).
\]
Note that $\norm{U}\leq \norm{\Phi}$. We claim that $U$ is an isomorphism onto $\mathrm{Im}(P)$, where $P=\Phi(1)$. For that, define an operator $V\colon\mathrm{Im}(P)\to \ell_p(X)$ by letting \[V\eta=\Phi^{-1}(\eta\otimes f)\delta_{x_0}\text{ for all }\eta\in \mathrm{Im}(P).\]
In order to show that $V$ is well-defined, note that $\eta\otimes f\in \mathrm{Im}(\Phi)$ for all $\eta\in\mathrm{Im}(P)$. Indeed, this follows since  $\Phi(\lproe(X))$ is hereditary in $\lproe(Y)$ and $f(\zeta)\eta\otimes f=P(\eta\otimes f)(\zeta\otimes f)$.  Moreover, since $\Phi(\xi\otimes \delta_{x_0})\zeta=P\Phi(\xi\otimes \delta_{x_0})\zeta$, it follows that $\mathrm{Im}(U)\subset \mathrm{Dom}(V)$. It is straightforward to check that  $V=U^{-1}$ and $\|V\|\leq \|\Phi^{-1}\|$. In particular, if $\Phi$ is an isometry, so is $U$.

We are left to show that $\Phi(a)=UaU^{-1}P$ for all $a\in \lproe(X)$.	For each $\xi\in \ell_p$ and each $g\in \ell_p(Y)^*$, we have
\[g\Big(\Phi(a)U\xi\Big) =g\Big(\Phi(a)\Phi(\xi\otimes\delta_{x_0})\zeta\Big)=
g\Big(\Phi(a\xi\otimes\delta_{x_0})\zeta\Big)=g\Big(Ua\xi\Big),\]
so  $\Phi(a)=\Phi(a)UU^{-1}P=UaU^{-1}P$ for all $a\in \lproe(X)$.
Since the product of an operator and a compact operator is compact, $\Phi$ is compact and we are done.
\end{proof}

\begin{lemma}\label{LemmaTheMapsAreCoarse}
Let $p\in [1,\infty)$, and $X$ and $Y$ be u.l.f. metric spaces. Let $\Phi\colon \lproe(X)\to \lproe(Y)$ be a  strongly continuous embedding onto a hereditary subalgebra of $\lproe(Y)$. Then for all $r,\delta>0$  there exists $s>0$ such that if  \[
d(x_1,x_2)\leq r\text{ and }\|\Phi(e_{x_1x_1})e_{y_1y_1}\|, \|e_{y_2y_2}\Phi(e_{x_2x_2})\|\geq \delta,
\]
 then $\partial (y_1,y_2)\leq s$.
\end{lemma}

\begin{proof}
Suppose otherwise. Then there exist 
 $r,\delta>0$, 
 sequences $(x^1_n)_n$ and $(x^2_n)_n$ in $X$, and sequences $(y^1_n)_n$ and $(y^2_n)_n$  in $Y$ such that $d(x_n^1,x_n^2)\leq r$, 
$\|\Phi(e_{x_n^1x_n^1})e_{y_n^1y_n^1}\|\geq \delta$,  $\|e_{y_n^2y_n^2}\Phi(e_{x_n^2x_n^2})\|\geq \delta$, and 
$\partial(y_n^1,y_n^2)\geq n$ for all $n\in\N$.

By Lemma \ref{LemmaPhiStronglyContAndU} and Proposition \ref{PropCoarseLike},  $\Phi\colon \lproe(X)\to\lproe(Y)$ is   coarse-like, so there exists $s>0$ such that $\partial(y,y')\geq s$ implies \[\|e_{y'y'}\Phi(e_{x^1_nx^2_n})e_{yy}\|<\delta^2\] for all $n\in\N$. Fix $n\in\N$ such that $\partial(y^1_n,y^2_n)> s$. 
 
\begin{claim}\label{claim:rank1}
Given $x^1,x^2\in X$ and $y^1,y^2\in Y$, we have\footnote{This claim can be compared with its Hilbert space version \cite[Lemma~6.5]{BragaFarahVignati2018}.}

\[\|e_{y^2y^2}\Phi(e_{x^1x^2})e_{y^1y^1}\|=\|e_{y^2y^2}\Phi(e_{x^2x^2})\|\cdot\|\Phi(e_{x^1x^1})e_{y^1y^1}\|\]
\end{claim} 

\begin{proof}
Let $P=\Phi(1)$. By Lemma \ref{LemmaPhiStronglyContAndU}, there exists a surjective isomorphism $U\colon\ell_p(X)\to P\ell_p(Y)$   such that $\Phi(a)=UaU^{-1}P$. We have
\begin{align*}
\delta_{y^2}\Big(\Phi(e_{x^1 x^2 })\delta_{y^1 }\Big)&=\delta_{y^2}\Big(Ue_{x^1 x^2 }U^{-1}P\delta_{y^1 }\Big)\\
&=U^*\delta_{y^2}\Big(e_{x^1 x^2 }U^{-1}P\delta_{y^1 }\Big)\\
&=U^*\delta_{y^2}\Big(\delta_{x^1}\big(U^{-1}P\delta_{y^1 }\big)\delta_{x^2}\Big)\\
&=\delta_{x^1}\Big(U^{-1}P\delta_{y^1 }\Big)\cdot \delta_{y^2}\Big(U \delta_{x^2}\Big).
\end{align*}
Therefore, since $\|e_{y^2 y^2 }\Phi(e_{x^1 x^2 })e_{y^1 y^1}\|=|\delta_{y^2} (\Phi(e_{x^1 x^2 })\delta_{y^1 })|$, $\|\Phi(e_{x^1 x^1 })e_{y^1 y^1 }\|=|\delta_{x^1}(U^{-1}P\delta_{y^1 })|$ and $\|e_{y^2 y^2 }\Phi(e_{x^2 x^2 })\|=|\delta_{y^2}(U \delta_{x^2})|$, we are done. 
\end{proof}

The contradiction comes from the previous claim, as it gives that
 \[
 \delta^2\leq \|e_{y^2_ny^2_n}\Phi(e_{x^2_nx^2_n})\|\cdot\|\Phi(e_{x^1_nx^1_n})e_{y^1_ny^1_n}\|
 =\|e_{y^2_ny^2_n}\Phi(e_{x^1_nx^2_n})e_{y^1_ny^1_n}\|<\delta^2.\qedhere
\]
\end{proof}

 \begin{lemma}\label{LemmaTheMapsAreExpanding}
Let $p\in[1,\infty)$, and $X$ and $Y$ be u.l.f. metric spaces. Let $\Phi\colon \lproe(X)\to \lproe(Y)$ be an  embedding onto a hereditary Banach subalgebra of $\lproe(Y)$.  Then for all $s,\delta>0$  there exists $r>0$ such that  if
\[
d(x_1,x_2)\geq r\text{ and }\|\Phi(e_{x_1x_1})e_{y_1y_1}\|, \|e_{y_2y_2}\Phi(e_{x_2x_2})\|\geq \delta,
\]
 then $\partial (y_1,y_2)\geq s$.
\end{lemma}

\begin{proof}
Suppose otherwise. Then there exist 
 $s,\delta>0$, 
 sequences $(x^1_n)_n$ and $(x^2_n)_n$ in $X$, and sequences $(y^1_n)_n$ and $(y^2_n)_n$  in $Y$ such that $d(x_n^1,x_n^2)\geq n$, 
$\|\Phi(e_{x_n^1x_n^1})e_{y_n^1y_n^1}\|\geq \delta$,  $\|e_{y_n^2y_n^2}\Phi(e_{x_n^2x_n^2})\|\geq \delta$, and 
$\partial(y_n^1,y_n^2)\leq s$ for all $n\in\N$. 

Since $d(x_n^1,x_n^2)\geq n$ for all $n\in\N$, by going to a subsequence, we can assume that  $(x^1_n)_n$ is a sequence of distinct elements (if not, exchange $x_n^1$ and $x_n^2$ and the proof will follow similarly).

\begin{claim}
We can assume that both $(y^1_n)_n$ and $(y^2_n)_n$ are sequences of distinct elements.
\end{claim}

\begin{proof}
If not, by going to a subsequence and using that $Y$ is u.l.f., since $\partial(y_n^1,y_n^2)\leq s$ for all $n\in\N$, we can assume that both $(y^1_n)_n$ and $(y^2_n)_n$ are constant. As $(x^1_n)_n$ is a sequence of distinct elements, $\sum_ne_{x^1_nx^1_n}$ converges in the strong operator topology to an element in $\lproe(X)$. Since $\Phi$ is strongly continuous by Lemma \ref{LemmaPhiStronglyContAndU}, we have that $\Phi(\sum_ne_{x^1_nx^1_n})=\sum_n\Phi(e_{x^1_nx^1_n})$. Hence, $\sum_n\Phi(e_{x^1_nx^1_n})\delta_{y^1_n}$ converges in $\ell_p(Y)$. However,  
\[
\|\Phi(e_{x^1_nx^1_n})\delta_{y^1_n}\|=\|\Phi(e_{x^1_nx^1_n})e_{y^1_ny^1_n}\|\geq \delta
\]
 for all $n\in\N$;  contradiction.
\end{proof}
 By going to a further subsequence, Lemma \ref{L:AA} and the fact that    $(\Phi(e_{x^1_nx^1_n}))_n$ is a  sequence  in $M_\infty^p(Y)$ allow us to assume that 
\[\|e_{y^1_my^2_m}\Phi(e_{x^1_nx^1_n})\|<  2^{-n-m-1}{\delta^2}\|\Phi\|^{-1}\]
for all $n\neq m$. 

Since $\partial(y^1_n,y^2_n)\leq s$ for all $n\in\N$, $\sum_{n\in\N}e_{y^1_ny^2_n}$ converges in the strong operator topology to an element in $\lproe(Y)$. As $\Phi(\lproe(X))$ is hereditary, there exists $a\in \lproe(X)$ such that 
\[\Phi(a)=\Phi(1)\Big(\sum_{n\in\N}e_{y^1_ny^2_n}\Big)\Phi(1).\]

\begin{claim}\label{Claim1234}
$\inf_n\|e_{x^2_nx^2_n}ae_{x^1_nx^1_n}\|\geq \delta^2/(2\|\Phi\|)$.
\end{claim}

\begin{proof}
First notice that for each $n\in\N$, we have
\[\|\Phi(e_{x^2_nx^2_n})e_{y^1_ny^2_n} \Phi(e_{x^1_nx^1_n})\|=
\|\Phi(e_{x^2_nx^2_n})e_{y^2_ny^2_n}\|\cdot\|e_{y^1_ny^1_n} \Phi(e_{x^1_nx^1_n})\| .\] Indeed, pick $\rho<\|e_{y^1_ny^1_n} \Phi(e_{x^1_nx^1_n})\| $ and a unit vector $\xi\in \ell_p(Y)$ such that $\|e_{y^1_ny^1_n} \Phi(e_{x^1_nx^1_n})\xi\|>\rho$. Hence, since $e_{y^1_ny^1_n} \Phi(e_{x^1_nx^1_n})\xi=\lambda\delta_{y^1_n}$, where $|\lambda|=\|e_{y^1_ny^1_n} \Phi(e_{x^1_nx^1_n})\xi\|$, we have that  
\[\|\Phi(e_{x^2_nx^2_n})e_{y^1_ny^2_n} \Phi(e_{x^1_nx^1_n})\xi\|=\|\Phi(e_{x^2_nx^2_n})e_{y^1_ny^2_n}\lambda\delta_{y^1_n}\|=|\lambda|\cdot \|\Phi(e_{x^2_nx^2_n})e_{y^2_ny^2_n}\|.\]
So, $\|\Phi(e_{x^2_nx^2_n})e_{y^1_ny^2_n} \Phi(e_{x^1_nx^1_n})\|\geq \rho\|\Phi(e_{x^2_nx^2_n})e_{y^2_ny^2_n}\|$. 

Hence, $\|\Phi(e_{x^2_nx^2_n})e_{y^1_ny^2_n} \Phi(e_{x^1_nx^1_n})\|\geq \delta^2$, and it follows that,  for all $n\in\N$, we have  
\begin{align*}
\|\Phi\|&\cdot \|e_{x^2_nx^2_n}  ae_{x^1_nx^1_n}\|
\geq \Big\|\Phi(e_{x^2_nx^2_n})\Phi(a) \Phi(e_{x^1_nx^1_n})\Big\|\\
&=\Big\|\Phi(e_{x^2_nx^2_n})\Big(\sum_{m\in\N} e_{y^1_my^2_m}\Big) \Phi(e_{x^1_nx^1_n})\Big\|\\
&\geq \|\Phi(e_{x^2_nx^2_n})e_{y^1_ny^2_n} \Phi(e_{x^1_nx^1_n})\|- \sum_{m\neq n}\|\Phi(e_{x^2_nx^2_n})e_{y^1_my^2_m} \Phi(e_{x^1_nx^1_n})\| \geq \frac{\delta^2}{2}.\qedhere
\end{align*}
\end{proof}
Since $a\in \lproe(X)$ and $\lim_nd(x^1_n,x^2_n)=\infty$, Claim \ref{Claim1234} gives us a contradiction.
\end{proof}

\begin{proof}[Proof of Theorem \ref{ThmEmbHereditarySubAlg}]
By Lemma \ref{LemmaPhiStronglyContAndU} and Lemma \ref{LemmaPickMapf}, there exist  maps $f\colon X\to Y$ and $g\colon X\to Y$ such that 
\[\|e_{g(x)g(x)}\Phi(e_{xx})e_{f(x)f(x)}\|\geq \delta\]
for all $x\in X$. In particular, $\|e_{g(x)g(x)}\Phi(e_{xx})\| \geq \delta$ and $\|\Phi(e_{xx})e_{f(x)f(x)}\|\geq \delta$ for all $x\in X$.
By Lemma \ref{LemmaCloseAndUnfFiniteToOne}\eqref{ItemCloseAndUnfFiniteToOne1}, $f$ and $g$ are close, so there is  $m>0$ such that  $\partial(f(x),g(x))\leq m$ for all $x\in X$.

We want to show that $f$ is coarse and expanding. Fix $r>0$, and let $s>0$ be given by Lemma \ref{LemmaTheMapsAreCoarse} for $r$ and $\delta$. Then 
\[
\partial(f(x),f(x'))\leq \partial (f(x),g(x'))+\partial (g(x'),f(x'))\leq s+m
\]
for all $x,x'\in X$ with $d(x,x')\leq r$, hence $f$ is coarse. To see that $f$ is expanding, fix $r>0$, and let $s>0$ be given by Lemma \ref{LemmaTheMapsAreExpanding} for $r$ and $\delta$. Then 
\[
\partial(f(x),f(x'))\geq \partial (f(x),g(x'))-\partial (g(x'),f(x'))\geq r-m
\]
for all $x,x'\in X$ with $d(x,x')\geq s$.
\end{proof}

\begin{proof}[Proof of Corollary~\ref{cor:onlybanach}]
   Fix an isomorphism $\Phi\colon\lproe(X)\to\lproe(Y)$. By Lemma \ref{LemmaPhiStronglyContAndU}, both $\Phi$ and $\Phi^{-1}$ are strongly continuous and compact preserving.  By Lemma~\ref{LemmaPickMapf} applied to $\Phi$ and $\Phi^{-1}$, pick  $\delta>0$, and maps  $f\colon X\to Y$ and $g\colon Y\to X$ with the property that
\[
\norm{\Phi(e_{xx})e_{f(x)f(x)}},\norm{\Phi^{-1}(e_{yy})e_{g(x)g(x)}}>\delta.
\]
for all $x\in X$ and all $y\in Y$. By Proposition \ref{PropCoarseLike}, both $\Phi$ and $\Phi^{-1}$ are coarse-like. Proceeding as in the proof of Theorem \ref{ThmEmbHereditarySubAlg},   both $f$  and $g$ are  coarse maps. Hence,  it is enough to show that $g\circ f$ is close to the identity on $X$ and that $f\circ g$ is close to the identity on $Y$. 

Suppose $g\circ f$ is not close to $\mathrm{Id}_X$. Then there exists a sequence $(x_n)_n\subseteq X$ of distinct points   such that $d(x_n,g(f(x_n)))\geq n$ for all $n\in\N$. For brevity let $y_n=f(x_n)$ and $z_n=g(y_n)$. By the choice of $f$ and $g$, we have that
\[
\norm{\Phi(e_{x_nx_n})e_{y_ny_n}},\norm{\Phi^{-1}(e_{y_ny_n})e_{z_nz_n}}>\delta
\]
for all $n\in\N$.   By Claim~\ref{claim:rank1} applied to $\Phi^{-1}$, we  have that
\begin{align*}
\norm{e_{x_nx_n}\Phi^{-1}(e_{y_ny_n})e_{z_nz_n}}&=\norm{e_{x_nx_n}\Phi^{-1}(e_{y_ny_n})}\cdot \norm{\Phi^{-1}(e_{y_ny_n})e_{z_nz_n}}\\
&\geq\|\Phi\|^{-1}\cdot \norm{\Phi(e_{x_nx_n})e_{y_ny_n}}\cdot \norm{\Phi^{-1}(e_{y_ny_n})e_{z_nz_n}} \\
&>\eta^2\|\Phi\|^{-1}
\end{align*}
 for every $n\in\NN$. Let $d$ be a metric on $X$. Since $\Phi^{-1}$ is coarse like and $\lim_nd(x_n,z_n)=\infty$, this gives us a contradiction. Similarly, we get that $f\circ g$ is close to $\mathrm{Id}_Y$. 
\end{proof}

\subsection{Property A and injectivity}\label{subsec:propA}
In this subsection, we show that if we
assume Y satisfies  the stronger geometric condition of property A, then the
coarse embedding of Theorem \ref{ThmEmbHereditarySubAlg}  can be taken to be injective and the coarse equivalence of Corollary \ref{cor:onlybanach} can be taken to be a bijective coarse equivalence.  The methods in this subsection are inspired by \cite{WhiteWillett2017}.


Throughout the remainder of this subsection, fix $p\in [1,\infty)$, u.l.f. metric spaces $(X,d)$ and $(Y,\partial)$ with property A, and a Banach algebra embedding $\Phi\colon \lproe(X)\to \lproe(Y)$ onto a hereditary subalgebra of $\lproe(Y)$. Moreover, by Lemma \ref{LemmaCloseAndUnfFiniteToOne}, fix an isomorphism   $U\colon \ell_p(X)\to \mathrm{Im}(\Phi(1))\subset \ell_p(Y)$ so that 
\[\Phi(a)=UaU^{-1}\Phi(1), \text{ for all }a\in \lproe(X).\]
At last, for $x\in X$, $A\subset X$ and $\delta>0$, we define
\[Y_{x,\delta}=\Big\{y\in Y\mid \exists z\in Y,\ \max\{\|e_{yy}\Phi(e_{xx})e_{zz}\|,\|e_{zz}\Phi(e_{xx})e_{yy}\|\}\geq \delta\Big\}\]  
and 
\[Y_{A,\delta}=\bigcup_{x\in A}Y_{x,\delta}.\] 

The next lemma is \cite[Lemma 6.6]{BragaFarahVignati2019} for $p\neq 2$. 

\begin{lemma}\label{LemmaUdelta}
For all $\gamma>0$ and all  $\eps>0$ there exists $r>0$ such that 
for all   $A\subseteq X$ and all $B\subseteq   Y$  with $\|\Phi(\chi_A)\chi_B\Phi(1)\|>\gamma$, there exists $D\subseteq   X$ with $\diam(D)<r$ such that \[\|\Phi(\chi_{A\cap D})\chi_{B}\Phi(1)\|\geq (1-\eps)\|\Phi(\chi_A)\chi_{B}\Phi(1)\|.\]
\end{lemma}

\begin{proof}[Sketch of the proof]
Since $Y$ has property A, it also has the \emph{$2$-operator norm localization  property}\footnote{We refer the reader to \cite[Defintiion 6.4]{WhiteWillett2017} for the definition of the \emph{$2$-operator norm localization property}. The \emph{$p$-operator norm localization property} (\emph{$p$-ONL} for short) is defined analogously.}   Precisely, the authors of \cite{ChenTesseraWangYu2008} introduced the \emph{metric sparsification property} (\cite[Definition 3.1]{ChenTesseraWangYu2008}) and proved that this property implies the $2$-ONL (see \cite[Proposition 3.3 and Proposition 4.1]{ChenTesseraWangYu2008}), and the metric sparsification property is now known to be equivalent to property A (see \cite[Proposition 3.2 and Theorem 3.8]{ChenTesseraWangYu2008} and \cite[Theorem 4.1]{Sako2014}). Moreover, the proof of \cite[Proposition 4.1]{ChenTesseraWangYu2008} works for any $p\in [1,\infty)$. Hence property A implies the $p$-ONL for all $p\in [1,\infty)$. The proof then follows completely analogously to the proof of \cite[Lemma 6.6]{BragaFarahVignati2019}, so we omit the details.
\end{proof}

\begin{lemma}\label{LemmaONL}
For all $\eps>0$ there is $\delta>0$ such that $\norm{\Phi(e_{xx})(1-\chi_{Y_{x,\delta}})}<\eps$ for all $x\in X$.
\end{lemma}
\begin{proof}
Assume not, and fix $\eps>0$ and a sequence $(x_n)_n\subseteq X$ such that
\[
\norm{\Phi(e_{x_nx_n})(1-\chi_{Y_{x_n,1/n}})}>\eps
\]
 for all $n\in\N$. For each $x\in X$, let $Z_x=\bigcup_{\delta>0} Y_{x,\delta}$, so $\Phi(e_{xx})\chi_{Z_x}=\Phi(e_{xx})$. Since $\Phi(e_{xx})$ is compact, we have that for every $x\in X$ there is $\delta>0$ such that $\norm{\Phi(e_{xx})(1-\chi_{Y_{x,\delta}})}<\eps/2$, therefore we can assume that all elements in $(x_n)_n$ are distinct.

By compactness of each $\Phi(e_{xx})$, there exists a sequence $(C_n)_n$ of finite subsets of $X$  such that 
\[
\norm{\chi_{C_n}\Phi(e_{x_nx_n})\chi_{C_n}(1-\chi_{Y_{x_n,1/n}})}>\eps/2
\]
for all $n\in\N$. By passing to a subsequence and thanks to Lemma~\ref{L:AA}, we can assume  that all the $C_n$'s are disjoint. For each $n\in\N$, let $A_n=C_n\cap (Y\setminus Y_{x_n,1/n})$ and let $a=\sum_n \chi_{C_n}\Phi(e_{x_nx_n})\chi_{A_n}$.  Since $C_n\cap C_m=\emptyset$ whenever $n\neq m$, $\sum_n \chi_{C_n}\Phi(e_{x_nx_n})\chi_{A_n}$ converges in the strongly operator topology, so $a$ is a well defined element of $\cB(\ell_p(Y))$. Moreover,  since $(C_n)_n$ are disjoint, $A_n\subseteq C_n$ and $\norm{\chi_{C_n}\Phi(e_{x_nx_n})\chi_{A_n}}>\eps/2$ for all $n\in\N$, it follows that $a$ is noncompact. 

Suppose   that $y,y'\in Y$ are such that $\norm{e_{yy}ae_{y'y'}}>1/n$. Since $A_n\subseteq C_n$ and the $C_n$'s are disjoint, we have that there is $i\in\N $ such that $y\in C_i$ and $y'\in A_i$. If $i>n$, then 
\[
1/i<1/n<\norm{e_{yy}ae_{y'y'}}=\norm{e_{yy}\Phi(e_{x_ix_i})e_{y'y'}},
\]
so $y'\in Y_{x_i,1/i}$, contradicting that $A_i\subseteq Y\setminus Y_{x_i,1/i}$. We have just proved that  $\norm{e_{yy}ae_{y'y'}}\leq 1/n$ for all $n\in\N$ and all  $y,y'\notin\bigcup_{m\leq n}C_m$. By finiteness of $\bigcup_{m\leq n}C_n$, $a$ is   a ghost.

Fix $\eta>0$. Since the map $\Phi$ is coarse-like, there is $s>0$ such that each $\Phi(e_{xx})$ can be $\eta$-$s$-approximated. Since cutting by characteristic functions in $\ell_\infty(Y)$ decreases the propagation, each $\chi_{C_n}\Phi(e_{x_nx_n})\chi_{A_n}$ can be $\eta$-$s$-approximated, and since the $C_n$ are disjoint, $a$ can be $\eta$-$s$-approximated. Since $\eta$ is arbitrary, $a\in\lproe(Y)$. This concludes the proof, leading to a contradiction to the space $Y$ having property $A$ (see Proposition \ref{PropPropAImpliesGhostsComp}).
\end{proof}

\begin{lemma}\label{LemmaCardinalityXBDelta}
For all $\eps>0$ there exists $\delta>0$ such that 
\[\|\Phi(\chi_A)(1-\chi_{Y_{B,\delta}})\Phi(1)\|<\eps,\]
for all finite subsets   $A,B\subseteq   X$ with $A\subset B$. In particular, if $\eps>0$, then   $|A|\leq |Y_{A,\delta}|$ for all   $A\subseteq   X$. 
\end{lemma}

\begin{proof}
Suppose not. Then there exists $\eps\in (0, 1)$ and  sequences $(A_n)_n$ and $(B_n)_n$ of finite subsets of $X$ such that $A_n\subset B_n$ and \[\|\Phi(\chi_{A_n})(1-\chi_{Y_{B_n,1/n}})\Phi(1)\|\geq \eps,\] for all $n\in\N$. Proceeding exactly as in \cite[Claim 6.9]{BragaFarahVignati2019}, we can assume that $(A_n)_n$ is a sequence of disjoint subsets.

 Let $W_n=Y\setminus Y_{B_n,\delta}$. Since $(A_n)_n$ is a disjoint sequence of finite subsets and $\|\Phi(\chi_{A_n})\chi_{W_n}\Phi(1)\|>\eps$ for all $n$, Lemma~\ref{L:AA} allows us to pick a sequence $(Y_n)_n$ of disjoint finite subsets of $Y$ such that $\|\Phi(\chi_{A_n})\chi_{W_n\cap	 Y_n}\Phi(1)\|>\eps/2$ for all $n\in\N$.  Therefore,  Lemma~\ref{LemmaUdelta} implies that there exists $r>0$ and a sequence $(D_n)_n$ of subsets of $X$ such that  $\diam(D_n)<r$ and 
\begin{equation}\label{EqONL1}
\|\Phi(\chi_{A_n\cap D_n})\chi_{W_n\cap Y_n}\Phi(1)\|\geq (1-\eps)\|\Phi(\chi_{A_n})\chi_{W_n\cap Y_n}\Phi(1)\|,\tag{$*$}
\end{equation}
for all $n\in\N$.

Since $X$ is u.l.f. and $\sup_n\diam (D_n)\leq r$, there exists $N\in\N$ so that $\sup_n|D_n|<N$. Pick $\theta>0$ small enough so that $2N\theta <\eps(1-\eps) $.
By Lemma \ref{LemmaONL}, pick $n\in \N$ large enough so that $\|\Phi(e_{xx})(1-\chi_{Y_{x,1/n}})\|<\theta$ for all $x\in X$. Then
\begin{align*}\label{EqONL3}
\|\Phi(\chi_{A_n\cap D_n })\chi_{W_n\cap Y_n}\Phi(1)\|&\leq\|\Phi(\chi_{A_n\cap D_n })\chi_{W_n}\|\tag{$**$}\\ 
&\leq \sum_{x\in A_n\cap D_n}\|\Phi(e_{xx})(1-\chi_{Y_{x,1/n}})\|\\
&\leq \theta|D_n|\\
&\leq \frac{\eps(1-\eps)}{2},
\end{align*}
for all $n\in\N$. Therefore, inequalities \eqref{EqONL1} and \eqref{EqONL3} imply that
  \[\|\Phi(\chi_{A_n})\chi_{B_n\cap Y_n}\Phi(1)\|\leq \frac{\eps}{2}\] for all $n\in\N$; contradiction.

We are left to show that, if $\eps>0$ then our choice of $\delta$ implies that $|A|\leq |Y_{A,\delta}|$ for all $A\subseteq  X$. Fix $\eps>0$ and let $\delta>0$  be given by the first statement of the lemma. Notice that $|A|=\rank \Phi(\chi_A)$ and $|Y_{A,\delta}|=\rank( \chi_{Y_{A,\delta}})$. Suppose $\rank( \Phi(\chi_A))>\rank( \chi_{Y_{A,\delta}})$.  Let $H=U(\ell_p(X))$, so $\Phi(1)$ is the projection onto $H$.  Since $\mathrm{corank}(1-\chi_{Y_{A,\delta}})=\rank \chi_{Y_{A,\delta}} $, we can pick
\[\xi\in \mathrm{Im}(1-\chi_{Y_{A,\delta}})\cap\mathrm{Im}(\Phi(\chi_A))\]
with norm 1. Since $\xi\in \mathrm{Im}(\Phi(\chi_A))\subseteq H$, this gives us that   $\|\Phi(\chi_A)( 1-\chi_{Y_{A,\delta}})\Phi(1)\xi\|=\|\xi\|=1$; contradiction.
\end{proof}

The next result follows from Lemma \ref{LemmaCardinalityXBDelta} and Hall's marriage theorem. Since the proof is completely analogous to the proof of \cite[Lemma 6.10]{BragaFarahVignati2019}, we omit it.

\begin{lemma}\label{LemmaExistenceInjection}
There exists $\delta>0$ and  an injection  $f\colon X\to Y$ such that  $f(y)\in Y_{x,\delta}$ for all $x\in X$. \qed
\end{lemma}

\begin{proof}[Proof of Theorem \ref{ThmEmbWithPropA}]
\eqref{Item1ThmEmbWithPropA} $\Rightarrow$ \eqref{Item2ThmEmbWithPropA} Let $f\colon X\to Y$ be an injective coarse embedding and define a map $V\colon \ell_p(X)\to \ell_p(Y)$ by letting $V\delta_x=\delta_{f(x)}$. Let $P=\chi_{\mathrm{Im}(f)}$, so $V^{-1}\colon \ell_p(\mathrm{Im}(f))\to \ell_p(X)$ is well defined. One can easily check that the map $a\in \lproe(X)\mapsto VaV^{-1}P\in \lproe(Y)$ is then an isometric embedding (cf. \cite[Theorem 8.1]{BragaFarah2018} and \cite[Theorem 1.4]{BragaFarahVignati2019}).

Since \eqref{Item2ThmEmbWithPropA} $\Rightarrow$ \eqref{Item3ThmEmbWithPropA} is straightforward, we are left to notice that \eqref{Item3ThmEmbWithPropA} $\Rightarrow$ \eqref{Item1ThmEmbWithPropA}. Let $\delta>0$   and $f\colon X\to Y$ be given by Lemma \ref{LemmaExistenceInjection}, so $f$ is injective. Let
\[X_1=\{x\in X\mid \exists z\in Y,\ \|e_{zz}\Phi(e_{xx})e_{f(x)f(x)}\|\geq \delta\}\]
and 
\[X_2=\{x\in X\mid \exists z\in Y,\ \|e_{f(x)f(x)}\Phi(e_{xx})e_{zz}\|\geq \delta\}\cap X_1^\complement.\]
Since $f(x)\in Y_{x,\delta}$ for all $x\in X$, we have that $X=X_1\sqcup X_2$. Let $g,h\colon X\to Y$ be maps so that $g\restriction X_1=f$, $h\restriction X_2=f$, $ \|e_{f(x)f(x)}\Phi(e_{xx})e_{g(x)g(x)}\|\geq \delta$ for all $x\in X_2$, and $ \|e_{h(x)h(x)}\Phi(e_{xx})e_{f(x)f(x)}\|\geq \delta$ for all $x\in X_1$. So,  
$ \|e_{h(x)h(x)}\Phi(e_{xx})e_{g(x)g(x)}\|\geq \delta$ for all $x\in X$, proceeding as in the proof of Theorem \ref{ThmEmbHereditarySubAlg}, we have  that $g$ (and $h$) is a coarse embedding. Since $\Phi$ is coarse-like, $f$ is close to $g$ (cf. Lemma \ref{LemmaCloseAndUnfFiniteToOne}), which gives that $f$ is a coarse embedding.
 \end{proof}

\begin{proof}[Proof of Theorem \ref{ThmIsomWithPropA}]
\eqref{Item1ThmIsomWithPropA} $\Rightarrow$ \eqref{Item2ThmIsomWithPropA} follows just as in the proof of Theorem \ref{ThmEmbWithPropA}, and  \eqref{Item2ThmEmbWithPropA} $\Rightarrow$ \eqref{Item3ThmEmbWithPropA} is straightforward.

\eqref{Item3ThmEmbWithPropA} $\Rightarrow$ \eqref{Item1ThmEmbWithPropA} 
Let $\delta>1$ and $f\colon X\to Y$ be  given by Lemma \ref{LemmaExistenceInjection}, so $f$ is injective. By symmetry, and replacing  $\delta$ by a smaller positive number if necessary,  Lemma \ref{LemmaExistenceInjection} give us a map $g\in Y\to X$ so that for all $y\in Y$, there exists $x\in X$ with 
 \[\max\{\|e_{xx}\Phi^{-1}(e_{yy})e_{g(y)g(y)}\|,\|e_{g(y)g(y)}\Phi^{-1}(e_{yy})e_{xx}\|\}\geq \delta.\]
Proceeding as in the proof of Theorem \ref{ThmEmbWithPropA}, we have that both $f$ and $g$ are coarse embeddings.
 
 K\'{o}nig's proof of the Cantor-Scr\"{o}der-Bernstein
theorem gives us a bijection $h\colon X\to Y$ such that for each $x\in X$, either
$h(x)=f(x)$ or $x\in \mathrm{Im}(g)$ and $h(x)=g^{-1}(x)$. A simple application of Lemma \ref{LemmaTheMapsAreCoarse}, Lemma \ref{LemmaTheMapsAreExpanding} and Lemma \ref{LemmaCloseAndUnfFiniteToOne} give us that $h$ is coarse and expanding, so we are done.\end{proof}

\begin{proof}[Proof of Corollary \ref{CorIsoAsCstarIFFasBanAlg}]
Suppose $\cstu(X)$ and $\cstu(Y)$ are isomorphic as Banach algebras. Since $Y$ has property A, so does $X$. Indeed, this follows since  property $A$ is equivalent to amenability of $\cstu(X)$ as a Banach algebra\footnote{By \cite[Theorem 5.3]{SkandalisTuYu2002} $X$ has property A if and only if $\cstu(X)$ is nuclear, and nuclearity is equivalent to amenability for \cstar-algebras.}, and Banach algebra amenability passes through isomorphisms. The result now  follows as the proof of Theorem \ref{ThmIsomWithPropA}. Indeed,  if $f\colon X\to Y$ is a bijective coarse equivalence, then $a\in \lproe(X)\mapsto VaV^{-1}\in \lproe(Y)$ is a \cstar-algebra isomorphism, where $V\colon \ell_p(X)\to \ell_p(Y)$ is given by  $V\delta_x=\delta_{f(x)}$ for all $x\in X$. 
\end{proof}

\end{document}